\documentclass[12pt,reqno]{amsart}
\usepackage{amsmath,amssymb,extarrows}
\usepackage{url}
\usepackage{tikz,enumerate}
\usepackage{diagbox}
\usepackage{appendix}
\usepackage{epic}

\usepackage{float} 

\usepackage{cite}

\usepackage{hyperref}
\usepackage{array}

\usepackage{booktabs}

\allowdisplaybreaks[4]

\newtheorem{theorem}{Theorem}
\newtheorem{corollary}[theorem]{Corollary}
\newtheorem{conj}[theorem]{Conjecture}
\newtheorem{lemma}[theorem]{Lemma}

\theoremstyle{definition}

\theoremstyle{remark}

\numberwithin{equation}{section}
\numberwithin{theorem}{section}
\numberwithin{defn}{section}

\begin{document}
\title[Counterexamples to Zagier's Duality Conjecture]
 {Counterexamples to Zagier's Duality Conjecture on Nahm Sums}

\author{Liuquan Wang}
\address{School of Mathematics and Statistics, Wuhan University, Wuhan 430072, Hubei, People's Republic of China}
\email{wanglq@whu.edu.cn;mathlqwang@163.com}



\begin{abstract}
Given any positive integer $r$, Nahm's problem is to determine all $r\times r$ rational positive definite matrix $A$, $r$-dimensional rational vector $B$ and rational scalar $C$ such that the rank $r$ Nahm sum associated with $(A,B,C)$ is modular.  Around 2007, Zagier conjectured that if the rank $r$ Nahm sum for $(A,B,C)$ is modular, then so is the dual Nahm sum associated with $(A^{-1},A^{-1}B,B^\mathrm{T} A^{-1}B/2-{r}/{24}-C)$.  We construct some explicit rank four Nahm sums which are modular while their duals are not modular. This provides counterexamples to Zagier's duality conjecture.
\end{abstract}

\maketitle

\section{Introduction and Main Results}\label{sec-intro}
A central problem linking the theory of $q$-series and modular forms is to determine the modularity of certain $q$-hypergeometric series. In this aspect, Nahm \cite{Nahm1994,Nahmconf,Nahm2007} considered a particular important class of series. Let $A$ be a $r\times r$ positive definite matrix, $B$ a $r$-dimensional column vector and $C$ a scalar. Nahm's problem  is to find all such $A,B,C$ with rational entries such that the \emph{Nahm sum}
\begin{align}\label{eq-Nahm}
    f_{A,B,C}(q):=\sum_{n=(n_1,\dots,n_r)^\mathrm{T} \in \mathbb{N}^r}\frac{q^{\frac{1}{2}n^\mathrm{T}An+n^\mathrm{T}B+C}}{(q;q)_{n_1} \cdots (q;q)_{n_r}}
\end{align}
is modular, and such $(A,B,C)$ is usually referred as a rank $r$ \emph{modular triple}. Here and throughout this paper we assume $|q|<1$ and use standard $q$-series notation:
\begin{align}
    (a;q)_n:=\prod_{k=0}^{n-1} (1-aq^k), \quad n \in \mathbb{N}\cup \{\infty\}.
\end{align}

The motivation of Nahm's problem comes from physics. Modular Nahm sums are expected to be characters of some 2-dimensional rational conformal field theories. Meanwhile, Nahm sums also appear as characters of standard modules of some affine Kac--Moody Lie algebras.

A prototype and important example of modular Nahm sums comes from the famous Rogers--Ramanujan identities:
\begin{align}
    \sum_{n=0}^\infty\frac{q^{n^2}}{(q;q)_n}
    =
    \frac{1}{(q,q^4;q^5)_\infty}, \quad
    \sum_{n=0}^\infty\frac{q^{n^2+n}}{(q;q)_n}
    =
    \frac{1}{(q^2,q^3;q^5)_\infty}.\label{RR}
\end{align}
These identities show that $(2,0,-1/60)$ and $(2,1,11/60)$ are two rank one modular triples. They motivate people to find similar sum-to-product $q$-series identities which are usually referred as Rogers--Ramanujan type identities. We refer the reader to Sills' book \cite{Sills-book} for an elaborate introduction on this topic.

Around 2007, Zagier \cite{Zagier} studied Nahm's problem systematically.
He stated explicitly  some conditions on $A$ in terms of polynomial equations and Bloch groups so that $A$ is the matrix part of a modular triple. This was known as Nahm's conjecture and is still open so far. When the rank $r=1$, Zagier proved this conjecture by showing that there are exactly seven modular triples. However, when the rank $r\geq 2$, Nahm's conjecture needs further modification since Vlasenko and Zwegers \cite{VZ}  provided two counterexamples. Nevertheless, Calegari, Garoufalidis and Zagier \cite{CGZ} confirmed one direction of Nahm's conjecture.

After an extensive search, Zagier \cite[Tables 2-3]{Zagier} discovered a number of modular triples in the rank two and rank three cases. The modularity of these examples have now all be confirmed by the works of Zagier \cite{Zagier}, Vlasenko--Zwegers \cite{VZ}, Cherednik--Feigin \cite{Feigin}, Calinescu--Milas--Penn \cite{CMP}, Cao--Rosengren--Wang \cite{CRW} and the author \cite{Wang-rank2,Wang-rank3}.

Zagier \cite[p.\ 50, (f)]{Zagier}  remarked the following important observation.
\begin{conj}\label{conj-dual}
If $(A,B,C)$ is a rank $r$ modular triple, then
\begin{align}\label{id-dual-Zagier}
(A^\star, B^\star, C^\star)=(A^{-1},A^{-1}B,\frac{1}{2}B^\mathrm{T} A^{-1}B-\frac{r}{24}-C)
\end{align}
is also a rank $r$ modular triple.
\end{conj}
This conjecture is supported by various evidences as discussed in \cite{Zagier}. First, the conditions on $A$ in Nahm's conjecture are invariant under $A\mapsto A^{-1}$. Second, the asymptotic behaviours of $f_{A,B,C}(q)$ and its dual series $f_{A^\star,B^\star,C^\star}(q)$ are closely related (see \cite[Eq.\ (38)]{Zagier}). Third, on the conformal field theory side, the involution $A\leftrightarrow A^{-1}$ is related to a duality found by Goddard--Kent--Olive \cite{GKO} and to the level-rank duality.
Moreover, since the dual of the rank two examples in \cite[Table 2]{Zagier} are contained in the same table, this  duality conjecture holds for these particular examples.

In 2023, Mizuno \cite{Mizuno} considered the following generalized Nahm sum:
\begin{align}\label{eq-general-Nahm-sum}
   \widetilde{f}_{A,B,C,D}(q):= \sum_{n=(n_1,\dots,n_r)^\mathrm{T} \in (\mathbb{Z}_{\geq 0})^r} \frac{q^{\frac{1}{2}n^\mathrm{T}ADn+n^\mathrm{T}B+C}}{(q^{d_1};q^{d_1})_{n_1}\cdots (q^{d_r};q^{d_r})_{n_r}}.
\end{align}
Here the notation is slightly rewritten, $D=\mathrm{diag}(d_1,\dots,d_r)$ ($d_1,\dots,d_r \in \mathbb{Z}_{>0}$),   $B \in \mathbb{Q}^r$ is a vector and  $C \in \mathbb{Q}$ is a scalar. Following \cite{Mizuno}, we call $A \in \mathbb{Q}^{r \times r} $ a symmetrizable matrix with the symmetrizer $D$ if $AD$ is symmetric positive definite. If $\widetilde{f}_{A,B,C,D}(q)$ is modular then we call $(A,B,C,D)$ a \emph{modular quadruple}.

Similar to Zagier's observation stated in Conjecture \ref{conj-dual}, Mizuno \cite[Conjecture 4.1]{Mizuno} proposed the following conjecture.
\begin{conj}\label{conj-Mizuno}
If $(A,B,C,D)$ is a modular quadruple, then $(A^\star,B^\star,C^\star,D^\star)$ is also modular quadruple where
\begin{align}\label{id-dual-Mizuno}
    A^\star=A^{-1}, \quad B^\star=A^{-1}B, \quad C^\star=\frac{1}{2}B^\mathrm{T} (AD)^{-1}B-\frac{\mathrm{tr} D}{24}-c, \quad D^\star=D.
\end{align}
\end{conj}
Clearly, when $D$ is the identity matrix, this is exactly Conjecture \ref{conj-dual}.

Following the numerical approach of Zagier \cite{Zagier}, Mizuno discovered a number of matrices $A$ of rank two and three as well as some corresponding vectors $B$, scalars $C$ and symmetrizers $D$ such that $\widetilde{f}_{A,B,C,D}(q)$ is modular. Besides two open conjectures in the rank two case and two exceptional rank three cases, the modularity of these quadruples $(A,B,C,D)$ have now been confirmed. In particular, Mizuno proved four examples in the rank two case, and the remaining cases were confirmed by Wang and the author \cite{WW23} except that some conjectural identities for two examples are open. The rank two examples in \cite[Table 1]{Mizuno} all support the truth of Conjecture \ref{conj-Mizuno}. The modularity of the rank three examples in \cite[Tables 2,3]{Mizuno} have all be proved by Wang and the author \cite{WW24-1,WW24-2} except that two particular quadruples in the fifth example of \cite[Table 3]{Mizuno} have been shown to be nonmodular. This example corresponds to the symmetrizer $D=\mathrm{diag}(2,2,1)$ and the matrix
\begin{align}\label{AB-nonmodular}
A=\begin{pmatrix}
2 & 2 & 2\\
2 & 4 & 4\\
1 & 2 & 3
\end{pmatrix}, \quad
AD=\begin{pmatrix}
4 & 4 & 2\\
4 & 8 & 4\\
2 & 4 & 3
\end{pmatrix}.
\end{align}
Mizuno stated ten vectors $B$ for this matrix, and the two nonmodular cases correspond to:
\begin{align}\label{Mizuno-exam5-vector}
B_1=
\begin{pmatrix}
    1\\0\\1/2
\end{pmatrix}, \quad B_2= \begin{pmatrix}
    3\\4\\5/2
\end{pmatrix}.
\end{align}
Wang and the author \cite[Theorem 1.4]{WW24-2} proved that
 \begin{align}
      &\sum_{i,j,k\ge 0}\frac{q^{\frac{3}{2}i^2+2j^2+4k^2+2ij+4ik+4jk+\frac{1}{2}i+j}}{(q;q)_i(q^2;q^2)_j(q^2;q^2)_k}=\frac{1}{4}\left(3\frac{(q^2;q^2)_\infty}{(q;q)_\infty}
      +\frac{(q;q)_\infty^3}{(q^2;q^2)_\infty}\right), \label{WW-id-1} \\
       &\sum_{i,j,k\ge 0}\frac{q^{\frac{3}{2}i^2+2j^2+4k^2+2ij+4ik+4jk+\frac{5}{2}i+3j+4k}}{(q;q)_i(q^2;q^2)_j(q^2;q^2)_k}=\frac{1}{4}q^{-1}\left(\frac{(q^2;q^2)_\infty}{(q;q)_\infty}
       -\frac{(q;q)_\infty^3}{(q^2;q^2)_\infty}  \right). \label{WW-id-2}
 \end{align}
This shows that $\widetilde{f}_{A,B_i,C,D}(q)$ ($i=1,2$) is not modular for any $C$.

Conjecture \ref{conj-Mizuno} motivates us to consider Nahm sums associated with the quadruples dual to \eqref{AB-nonmodular} and \eqref{Mizuno-exam5-vector} (with the same symmetrizer $D^\star=\mathrm{diag}(2,2,1)$):
\begin{equation}
\begin{split}
&A^\star=\begin{pmatrix} 1 & -1/2 & 0 \\
-1/2 & 1 & -1 \\
0 & -1/2  & 1
\end{pmatrix}, \quad   A^\star D^\star=\begin{pmatrix}
2 & -1 & 0 \\ -1 & 2 & -1 \\ 0 &-1 & 1
\end{pmatrix},   \\
&B_1^\star=\begin{pmatrix} 1 \\ -1 \\ 1/2 \end{pmatrix}, \quad B_2^\star=\begin{pmatrix} 1 \\ 0 \\ 1/2 \end{pmatrix}.
\end{split}
\end{equation}
We establish the following identities for these dual Nahm sums.
\begin{theorem}\label{thm-id}
We have
\begin{align}
\sum_{i,j,k\geq 0} \frac{q^{\frac{1}{2}i^2+j^2+k^2-ik-jk+\frac{1}{2}i+j-k}}{(q;q)_i(q^2;q^2)_j(q^2;q^2)_k}&=3\frac{(q^2;q^2)_\infty^3}{(q;q)_\infty^3}, \label{thm-id-1} \\
\sum_{i,j,k\geq 0} \frac{q^{\frac{1}{2}i^2+j^2+k^2-ik-jk+\frac{1}{2}i+j}}{(q;q)_i(q^2;q^2)_j(q^2;q^2)_k}&=\frac{(q^2;q^2)_\infty^3}{(q;q)_\infty^3}. \label{thm-id-2}
\end{align}
\end{theorem}
This leads to some surprising consequences. The identities \eqref{thm-id-1}--\eqref{thm-id-2} show that $\widetilde{f}_{A^\star,B_i^\star,1/8,D^\star}(q)$ ($i=1,2$) is modular though $\widetilde{f}_{A,B_i,C',D}(q)$ ($i=1,2$) is not modular for any $C'$. This provides counterexamples to  Conjecture \ref{conj-Mizuno}.

The key ingredient in our proof of the identities \eqref{thm-id-1} and \eqref{thm-id-2} is the following set of single-sum Rogers--Ramanujan type identities.
\begin{theorem}\label{thm-new-RR}
We have
\begin{align}
&\sum_{k=0}^\infty \frac{q^{k^2}(-1;q)_{2k}(-1;q^2)_k}{(q^2;q^2)_{2k}}=\frac{1}{2}\left(3\frac{(q^2;q^2)_\infty^3}{(q;q)_\infty^2(q^4;q^4)_\infty}-\frac{(q;q)_\infty^2(q^2;q^2)_\infty}{(q^4;q^4)_\infty}\right), \label{key-id-1} \\
&\sum_{k=0}^\infty \frac{q^{k^2+2k}(-1;q)_{2k}(-1;q^2)_k}{(q^2;q^2)_{2k}}=\frac{1}{2}\left(\frac{(q^2;q^2)_\infty^3}{(q;q)_\infty^2(q^4;q^4)_\infty}+\frac{(q;q)_\infty^2(q^2;q^2)_\infty}{(q^4;q^4)_\infty}\right),  \label{key-id-2} \\
&\sum_{k=0}^\infty \frac{q^{k^2+k}(-1;q)_{2k+1}(-q;q^2)_k}{(q^2;q^2)_{2k+1}}=\frac{1}{2}\left(3\frac{(q^4;q^4)_\infty}{(q;q)_\infty}+\frac{(q;q)_\infty^3(q^4;q^4)_\infty}{(q^2;q^2)_\infty^2}\right), \label{key-id-3} \\
&\sum_{k=0}^\infty \frac{q^{k^2+3k+1}(-1;q)_{2k+1}(-q;q^2)_k}{(q^2;q^2)_{2k+1}}=\frac{1}{2}\left(\frac{(q^4;q^4)_\infty}{(q;q)_\infty}-\frac{(q;q)_\infty^3(q^4;q^4)_\infty}{(q^2;q^2)_\infty^2}\right). \label{key-id-4}
\end{align}
\end{theorem}
When we transform the generalized Nahm sums in \eqref{WW-id-1}--\eqref{WW-id-2} and \eqref{thm-id-1}--\eqref{thm-id-2} into Nahm sums (of the form \eqref{eq-Nahm}), we find a further remarkable result.
\begin{theorem}\label{thm-counterexample}
The Nahm sums $f_{A,B_i,1/16}(q)$ ($i=1,2$) are modular for
\begin{align}\label{exam-data-1}
A=\begin{pmatrix} 1 & 0 & 0 & -1/2 \\
0 & 1 & 0 & -1/2 \\
0 & 0 & 1 & -1/2 \\
-1/2 & -1/2 & -1/2 & 1 \end{pmatrix},  \quad B_1=\begin{pmatrix} 0 \\ 1/2 \\ 1/2 \\ -1/2 \end{pmatrix}, \quad B_2=\begin{pmatrix} 0 \\ 1/2 \\ 1/2 \\ 0\end{pmatrix}.
\end{align}
However, the dual Nahm sums $f_{A^\star,B_i^\star,C'}(q)$ ($i=1,2$) are not modular for any $C'$ where
\begin{align}\label{exam-data-2}
A^\star=\begin{pmatrix} 2 & 1 & 1 & 2 \\
1 & 2 & 1 & 2 \\
1 & 1 & 2 & 2 \\
2 & 2 & 2 & 4
\end{pmatrix}, \quad  B_1^\star=\begin{pmatrix}
0 \\ 1/2 \\ 1/2 \\ 0 \end{pmatrix}, \quad B_2^\star=\begin{pmatrix} 1 \\ 3/2 \\ 3/2 \\ 2\end{pmatrix}.
\end{align}
\end{theorem}
Theorem \ref{thm-counterexample} provides counterexamples to Conjecture \ref{conj-dual}.  Nevertheless, as can be seen from here and previous works such as \cite{WW23,MW24}, there do exist some duality phenomenon between Nahm sums. Conjectures \ref{conj-dual} and \ref{conj-Mizuno} can help us to construct new possible modular triples and quadruples, respectively. Moreover, the dual Nahm sums in Theorem \ref{thm-counterexample} can be expressed as sums of two modular forms of different weights (see \eqref{dual-exp-1} and \eqref{dual-exp-2}), and hence are still within the framework of modular forms.  The correct formulation of these conjectures is left as an interesting open problem.

The rest of this paper is organized as follows. In Section \ref{sec-pre} we briefly review some basic $q$-series identities and the theory of Bailey pairs. We also establish four Bailey pairs which will be invoked in our proof of Theorem \ref{thm-new-RR}, and we present some single-sum Rogers--Ramanujan type identities as byproducts. We give proofs of Theorems \ref{thm-id}--\ref{thm-counterexample} in Section \ref{sec-proof}.

\section{Preliminaries}\label{sec-pre}
We need one of Euler's $q$-exponential identities  \cite[Corollary 2.2]{Andrews}:
\begin{align}
\sum_{n=0}^{\infty}\frac{q^{(^n_2)} z^n}{(q;q)_n}
&=
(-z;q)_{\infty}. \label{Euler-2}
\end{align}

Recall one of Ramanujan's theta functions (see \cite[Eq.\ (2.2.13)]{Andrews} or \cite[Eq.\ (1.3.14)]{Berndt-book}):
\begin{align}
    \psi(q)&:=\sum_{n=0}^\infty q^{n(n+1)/2}=\frac{(q^2;q^2)_\infty^2}{(q;q)_\infty}. \label{psi-defn}
\end{align}
Here the last equality follows from the Jacobi triple product identity \cite[Theorem 2.8]{Andrews}
\begin{align}\label{JTP}
(q,z,q/z;q)_\infty=\sum_{n=-\infty}^\infty (-1)^nq^{\binom{n}{2}}z^n.
\end{align}
As an important consequence of \eqref{JTP}, we have Jacobi's identity \cite[Theorem 1.3.9]{Berndt-book}:
\begin{align}\label{eq-Jacobi-id}
    (q;q)_\infty^3=\sum_{n=0}^\infty (-1)^n(2n+1)q^{n(n+1)/2}.
\end{align}

A pair of sequences $(\alpha_n(a;q),\beta_n(a;q))$ is called a Bailey pair relative to $a$ if for all $n\geq 0$,
 \begin{align}\label{defn-BP}
     \beta_n(a;q)=\sum_{k=0}^n\frac{\alpha_k(a;q)}{(q;q)_{n-k}(aq;q)_{n+k}}.
 \end{align}
Bailey's lemma (see e.g. \cite[p.\ 4]{MSZ2008}) implies some transformation formulas which are quite useful in proving Rogers--Ramanujan type identities. Here we recall two of such formulas (see e.g. \cite[(TBL), (S2BL)]{MSZ2008}). Suppose that $(\alpha_n(a;q), \beta_n(a;q))$ is a Bailey pair relative to $a \in \{1,q\}$, then
\begin{align}
\sum_{n=0}^\infty q^{n^2}(-q;q^2)_n\beta_n(1;q^2)&=\frac{1}{\psi(-q)}\sum_{r=0}^\infty q^{r^2}\alpha_r(1;q^2), \label{TBL} \\
\sum_{n=0}^\infty q^{n(n+1)/2}(-q;q)_n\beta_n(q;q)&= \frac{1-q}{\varphi(-q)}\sum_{r=0}^\infty q^{r(r+1)/2}\alpha_r(q;q). \label{S2BL}
\end{align}

Along the proof of Theorem \ref{thm-id}, we find four Bailey pairs which appear to be new.
\begin{lemma}\label{lem-new-BP}
Let $\alpha_0^{(i)}(a;q)=\beta_0^{(i)}(a;q)=1$ ($i=1,2$, $a=1,q$) and for $n\geq 1$ we define
\begin{align}
&\alpha_n^{(1)}(1;q):=(2n+1)q^{n(n-1)/2}-(2n-1)q^{n(n+1)/2}, \beta_n^{(1)}(1;q):=\frac{(-1;q)_n^2}{(q;q)_{2n}}; \label{BP-1} \\
&\alpha_n^{(2)}(1;q):=(2n+1)q^{n(n+1)/2}-(2n-1)q^{n(n-1)/2}, \beta_n^{(2)}(1;q):=\frac{q^n(-1;q)_n^2}{(q;q)_{2n}}; \label{BP-2} \\
&\alpha_n^{(1)}(q;q):=(n+1)q^{n^2/2}-nq^{(n+1)^2/2},  \beta_n^{(1)}(q;q):=\frac{(-q^{1/2};q)_n^2}{(q^2;q)_{2n}}; \label{BP-3} \\
&\alpha_n^{(2)}(q;q):=(n+1)q^{(n+1)^2/2}-nq^{n^2/2},  \beta_n^{(2)}(q;q):=\frac{q^{n+\frac{1}{2}}(-q^{1/2};q)_n^2}{(q^2;q)_{2n}}. \label{BP-4}
\end{align}
Then the above $(\alpha_n^{(i)}(a;q);\beta_n^{(i)}(a;q))$ are Bailey pairs relative to $a$.
\end{lemma}
\begin{proof}
(1) We rewrite the identity \cite[Eq.\ (4.3)]{Andrews2016} in the following equivalent form:
\begin{align}
\frac{(-x^{-1};q)_n(-xq;q)_n}{(q;q)_{2n}}=\frac{1}{(q;q)_n^2}+\sum_{r=1}^n\frac{q^{\binom{r+1}{2}}x^r+q^{\binom{r}{2}}x^{-r}}{(q;q)_{n-r}(q;q)_{n+r}}.
\label{even-finite}
\end{align}
Replacing $x$ by $x^{-1}$ and multiplying both sides by $x$, we obtain
\begin{align}
x\frac{(-x;q)_n(-x^{-1}q;q)_n}{(q;q)_{2n}}=\frac{x}{(q;q)_n^2}+\sum_{r=1}^n\frac{q^{\binom{r+1}{2}}x^{1-r}+q^{\binom{r}{2}}x^{r+1}}{(q;q)_{n-r}(q;q)_{n+r}}.
\label{even-finite-change}
\end{align}
Subtracting \eqref{even-finite-change} from \eqref{even-finite} and then dividing both sides by $(1-x)$, we deduce that
\begin{align}
&\frac{(-x;q)_n(-x^{-1};q)_n}{(q;q)_{2n}}\left(\frac{1+xq^n}{1+x}-x\frac{1+x^{-1}q^n}{1+x^{-1}} \right)\times \frac{1}{1-x} \\
&=\frac{1}{(q;q)_n^2}+\sum_{r=1}^n\frac{1}{(q;q)_{n-r}(q;q)_{n+r}} \times \left( q^{\binom{r+1}{2}} \frac{x^r-x^{1-r}}{1-x} +q^{\binom{r}{2}}\frac{x^{-r}-x^{r+1}}{1-x}\right). \nonumber
\end{align}
Simplifying and then letting $x\rightarrow 1$, we obtain
\begin{align}
\frac{(-1;q)_n^2}{(q;q)_{2n}}=\frac{1}{(q;q)_n^2}+\sum_{r=1}^n \frac{(2r+1)q^{\binom{r}{2}}-(2r-1)q^{\binom{r+1}{2}}}{(q;q)_{n-r}(q;q)_{n+r}}.
\end{align}
This confirms the Bailey pair in \eqref{BP-1}.

(2) Replacing $x$ by $x^{-1}$ in \eqref{even-finite} and then multiplying both sides by $x^{-1}$, we obtain
\begin{align}
x^{-1}\frac{(-x;q)_n(-x^{-1}q;q)_n}{(q;q)_{2n}}=\frac{x^{-1}}{(q;q)_n^2}+\sum_{r=1}^n\frac{q^{\binom{r+1}{2}}x^{-r-1}+q^{\binom{r}{2}}x^{r-1}}{(q;q)_{n-r}(q;q)_{n+r}}.
\label{even-finite-new}
\end{align}
Subtracting \eqref{even-finite-new} from \eqref{even-finite}, then dividing both sides by $(1-x)$ and letting $x\rightarrow 1$, we obtain
\begin{align}
\frac{q^n(-1;q)_n^2}{(q;q)_{2n}}=\frac{1}{(q;q)_n^2}+\sum_{r=1}^n \frac{(2r+1)q^{\binom{r+1}{2}}-(2r-1)q^{\binom{r}{2}}}{(q;q)_{n-r}(q;q)_{n+r}}.
\end{align}
This confirms the Bailey pair in \eqref{BP-2}.

(3) We rewrite the identity \cite[Eq.\ (4.9)]{Andrews2016} in the following equivalent form:
\begin{align}
\frac{(-x^{-1}q^{1/2};q)_n(-xq^{1/2};q)_{n+1}}{(q^2;q)_{2n}}=\sum_{r=0}^n \frac{q^{r^2/2}x^{-r}+q^{(r+1)^2/2}x^{r+1}}{(q;q)_{n-r}(q^2;q)_{n+r}}. \label{odd-finite}
\end{align}
Replacing $x$ by $x^{-1}$ and then multiplying both sides by $x^2$, we obtain
\begin{align}
x^2\frac{(-xq^{1/2};q)_n(-x^{-1}q^{1/2};q)_{n+1}}{(q^2;q)_{2n}}=\sum_{r=0}^n \frac{q^{r^2/2}x^{r+2}+q^{(r+1)^2/2}x^{1-r}}{(q;q)_{n-r}(q^2;q)_{n+r}}. \label{odd-finite-change}
\end{align}
Subtracting \eqref{odd-finite-change} from \eqref{odd-finite} and then dividing both sides by $(1-x^2)$, we deduce that
\begin{align}
&\frac{(-x^{-1}q^{1/2};q)_n(-xq^{1/2};q)_n}{(q^2;q)_{2n}}\times \frac{(1+xq^{n+\frac{1}{2}})-x^2(1+x^{-1}q^{n+\frac{1}{2}})}{1-x^2}   \\
&=\sum_{r=0}^n \frac{1}{(q;q)_{n-r}(q^2;q)_{n+r}} \times \left(q^{r^2/2}\frac{x^{-r}-x^{r+2}}{1-x^2}+ q^{(r+1)^2/2} \frac{x^{r+1}-x^{1-r}}{1-x^2}\right). \nonumber
\end{align}
Simplifying and then letting $x\rightarrow 1$, we confirm the Bailey pair in \eqref{BP-3}.

(4) Replacing $x$ by $x^{-1}$ in \eqref{odd-finite}, then subtracting the resulting identity from \eqref{odd-finite}, dividing both sides by $(1-x^2)$ and letting $x\rightarrow 1$, we obtain
\begin{align}
q^{n+\frac{1}{2}}\frac{(-q^{1/2};q)_n^2}{(q^2;q)_{2n}}=\sum_{r=0}^n \frac{(r+1)q^{(r+1)^2/2}-rq^{r^2/2}}{(q;q)_{n-r}(q^2;q)_{n+r}}.
\end{align}
This confirms the Bailey pair in \eqref{BP-4}.
\end{proof}

If we plug the Bailey pairs \eqref{BP-1}--\eqref{BP-4} into some transformation formulas, we will get various identities as consequences. For example, if we substitute them into the formula (see e.g. \cite[Eq.\ (1.2.8)]{MSZ2008}):
\begin{align}\label{BP-tran-128}
\sum_{n=0}^\infty a^nq^{n^2}\beta_n(a;q)=\frac{1}{(aq;q)_\infty} \sum_{r=0}^\infty a^rq^{r^2} \alpha_r(a;q),
\end{align}
we obtain the following interesting identities though they are not needed in proof of the theorems stated before. For convenience we denote $J_m:=(q^m;q^m)_\infty$.
\begin{corollary}\label{cor-RR}
We have
\begin{align}
\sum_{n=0}^\infty \frac{q^{n^2}(-1;q)_n^2}{(q;q)_{2n}}&=\frac{4}{3}\frac{J_2J_3^2}{J_1^2J_6}-\frac{1}{3}\frac{J_1^4}{J_2^2}, \label{cor-id-1} \\
\sum_{n=0}^\infty \frac{q^{n^2+n}(-1;q)_n^2}{(q;q)_{2n}}&=\frac{1}{3}\frac{J_1^4}{J_2^2}+\frac{2}{3}\frac{J_2J_3^2}{J_1^2J_6}, \label{cor-id-2} \\
\sum_{n=0}^\infty \frac{q^{2n^2+2n}(-q;q^2)_n^2}{(q^2;q^2)_{2n+1}}
&=\frac{1}{3}\frac{J_1^2J_4^2}{J_2^2}+\frac{2}{3}\frac{J_2J_3J_{12}}{J_1J_4J_6},  \label{cor-id-3} \\
\sum_{n=0}^\infty \frac{q^{2n^2+4n}(-q;q^2)_n^2}{(q^2;q^2)_{2n+1}}&=-\frac{1}{3}q^{-1}\frac{J_1^2J_4^2}{J_2^2}+\frac{1}{3}q^{-1}\frac{J_2J_3J_{12}}{J_1J_4J_6}.  \label{cor-id-4}
\end{align}
\end{corollary}
\begin{proof}
We first recall two identities (see e.g. \cite[Corollaries 1.3.21 and 1.3.22]{Berndt-book}):
\begin{align}
    \sum_{n=-\infty}^\infty (6n+1)q^{(3n^2+n)/2}&=\frac{(q;q)_\infty^5}{(q^2;q^2)_\infty^2}, \label{Jacobi-cor-1} \\
    \sum_{n=-\infty^\infty} (3n+1)q^{3n^2+2n}&=\frac{(q;q)_\infty^2(q^4;q^4)_\infty^2}{(q^2;q^2)_\infty}. \label{Jacobi-cor-2}
\end{align}
Substituting \eqref{BP-1} into \eqref{BP-tran-128}, we deduce that
\begin{align}
&\sum_{n=0}^\infty \frac{q^{n^2}(-1;q)_n^2}{(q;q)_{2n}}=\sum_{n=0}^\infty q^{n^2}\beta_n^{(1)}(1;q)\nonumber \\
&=\frac{1}{(q;q)_\infty} \Big(1+\sum_{r=1}^\infty \big((2r+1)q^{(3r^2-r)/2}-(2r-1)q^{(3r^2+r)/2}\big)\Big) \nonumber \\
&=-\frac{1}{(q;q)_\infty} \sum_{r=-\infty}^\infty (2r-1)q^{(3r^2+r)/2} \nonumber \\
&=\frac{1}{(q;q)_\infty}\times \frac{1}{3}\Big(4\sum_{r=-\infty}^\infty q^{(3r^2+r)/2}-(6r+1)q^{(3r^2+r)/2}   \Big) \nonumber \\
&=\frac{4}{3}\frac{J_2J_3^2}{J_1^2J_6}-\frac{1}{3}\frac{J_1^4}{J_2^2}.\label{lem-S1-start}
\end{align}
Here for the last equality we used \eqref{Jacobi-cor-1} and \eqref{JTP}. This proves \eqref{cor-id-1}.

Similarly, substituting \eqref{BP-2} into \eqref{BP-tran-128}, we have
\begin{align}
&\sum_{n=0}^\infty \frac{q^{n^2+n}(-1;q)_n^2}{(q;q)_{2n}}=\sum_{n=0}^\infty q^{n^2}\beta_n^{(2)}(1;q) \nonumber \\
&=\frac{1}{(q;q)_\infty} \Big(1+\sum_{r=1}^\infty (2r+1)q^{(3r^2+r)/2}-\sum_{r=1}^\infty (2r-1)q^{(3r^2-r)/2}  \Big) \nonumber \\
&=\frac{1}{(q;q)_\infty} \sum_{r=-\infty}^\infty (2r+1)q^{(3r^2+r)/2} \nonumber \\
&=\frac{1}{(q;q)_\infty}\times \frac{1}{3}\Big(\sum_{r=-\infty}^\infty (6r+1)q^{(3r^2+r)/2}+2\sum_{r=-\infty}^\infty q^{(3r^2+r)/2}\Big) \nonumber \\
&=\frac{1}{3}\frac{J_1^4}{J_2^2}+\frac{2}{3}\frac{J_2J_3^2}{J_1^2J_6}.
\end{align}
Here for the last equality we used \eqref{Jacobi-cor-1} and \eqref{JTP}. This proves \eqref{cor-id-2}.

Next, substituting \eqref{BP-3} into \eqref{BP-tran-128}, we have
\begin{align}
&\sum_{n=0}^\infty \frac{q^{2n^2+2n}(-q;q^2)_n^2}{(q^2;q^2)_{2n+1}}=\frac{1}{1-q^2}\times \sum_{n=0}^\infty q^{2n^2+2n}\beta_n^{(1)}(q^2;q^2) \nonumber \\
&=\frac{1}{(q^2;q^2)_\infty} \Big(1+\sum_{r=1}^\infty \big( (r+1)q^{3r^2+2r}-rq^{(r+1)(3r+1)}\big)\Big) \nonumber \\
&=\frac{1}{(q^2;q^2)_\infty} \sum_{r=-\infty}^\infty (r+1)q^{3r^2+2r} \nonumber \\
&=\frac{1}{3(q^2;q^2)_\infty} \times \Big(\sum_{r=-\infty}^\infty (3r+1)q^{3r^2+2r}+2\sum_{r=-\infty}^\infty q^{3r^2+2r}  \Big) \nonumber \\
&=\frac{1}{3}\frac{J_1^2J_4^2}{J_2^2}+\frac{2}{3}\frac{J_2J_3J_{12}}{J_1J_4J_6}.
\end{align}
Here for the last equality we used \eqref{Jacobi-cor-2} and \eqref{JTP}. This proves \eqref{cor-id-3}.

Similarly, substituting \eqref{BP-4} into \eqref{BP-tran-128}, we have
\begin{align}
&\sum_{n=0}^\infty \frac{q^{2n^2+4n}(-q;q^2)_n^2}{(q^2;q^2)_{2n+1}}=\frac{q^{-1}}{1-q^2}\times \sum_{n=0}^\infty q^{2n^2+4n}\beta_n^{(2)}(q^2;q^2) \nonumber \\
&=\frac{q^{-1}}{(q^2;q^2)_\infty} \Big(1+\sum_{r=1}^\infty \big( (r+1)q^{(r+1)(3r+1)}-rq^{3r^2+2r}\big)\Big) \nonumber \\
&=-\frac{q^{-1}}{(q^2;q^2)_\infty} \sum_{r=-\infty}^\infty rq^{3r^2+2r} \nonumber \\
&=-\frac{q^{-1}}{3(q^2;q^2)_\infty} \times \Big(\sum_{r=-\infty}^\infty (3r+1)q^{3r^2+2r}-\sum_{r=-\infty}^\infty q^{3r^2+2r}  \Big) \nonumber \\
&=-\frac{1}{3}q^{-1}\frac{J_1^2J_4^2}{J_2^2}+\frac{1}{3}q^{-1}\frac{J_2J_3J_{12}}{J_1J_4J_6}.
\end{align}
Here for the last equality we used \eqref{Jacobi-cor-2} and \eqref{JTP}. This proves \eqref{cor-id-4}.
\end{proof}

\section{Proofs of the Theorems}\label{sec-proof}

\begin{proof}[Proof of Theorem \ref{thm-new-RR}]
Substituting the Bailey pair \eqref{BP-1} into \eqref{TBL}, we deduce that
\begin{align}
&\sum_{n=0}^\infty \frac{q^{n^2}(-1;q)_{2n}(-1;q^2)_n}{(q^2;q^2)_{2n}} =\sum_{n=0}^\infty q^{n^2}(-q;q^2)_n\beta_n^{(1)}(1;q^2) \nonumber \\
&=\frac{1}{\psi(-q)}\sum_{r=0}^\infty q^{r^2}\alpha_r^{(1)}(1;q^2) \nonumber \\
&=\frac{1}{\psi(-q)}\Big(1+\sum_{r=1}^\infty (2r+1)q^{2r^2-r}-\sum_{r=1}^\infty (2r-1)q^{2r^2+r} \Big) \nonumber \\
&=\frac{1}{2}\times \frac{(q^2;q^2)_\infty}{(q;q)_\infty(q^4;q^4)_\infty} \sum_{n=0}^\infty \big(3-(-1)^n(2n+1)\big)q^{n(n+1)/2}.
\end{align}
Utilizing the identity \eqref{psi-defn} and \eqref{eq-Jacobi-id}, we obtain \eqref{key-id-1}.

Similarly, substituting the Bailey pair \eqref{BP-2} into \eqref{TBL}, we obtain \eqref{key-id-2}.

Substituting the Bailey pair \eqref{BP-3} into \eqref{S2BL}, we deduce that
\begin{align}
&\sum_{n=0}^\infty \frac{q^{n^2+n}(-1;q)_{2n+1}(-q;q^2)_n}{(q^2;q^2)_{2n+1}} =\frac{2}{1-q^2}\times \sum_{n=0}^\infty q^{n^2+n}(-q^2;q^2)_n \beta_n^{(1)}(q^2;q^2) \nonumber \\
&=\frac{2}{\varphi(-q^2)} \sum_{r=0}^\infty q^{r(r+1)} \alpha_r^{(1)}(q^2;q^2) \nonumber \\
&=2\frac{(q^4;q^4)_\infty}{(q^2;q^2)_\infty^2} \sum_{r=0}^\infty \Big((r+1)q^{2r^2+r}-rq^{(r+1)(2r+1)}\Big) \nonumber \\
&=\frac{1}{2}\times \frac{(q^4;q^4)_\infty}{(q^2;q^2)_\infty^2} \sum_{n=0}^\infty \Big((-1)^n(2n+1)+3 \Big)q^{n(n+1)/2}.
\end{align}
Utilizing the identity \eqref{psi-defn} and \eqref{eq-Jacobi-id}, we obtain \eqref{key-id-3}.

Similarly, substituting the Bailey pair \eqref{BP-4} into \eqref{S2BL}, we obtain \eqref{key-id-4}.
\end{proof}

\begin{proof}[Proof of Theorem \ref{thm-id}]
We define
\begin{align}\label{F-defn}
F(u,v,w)=F(u,v,w;q):=\sum_{i,j,k\geq 0} \frac{q^{\frac{1}{2}i^2+j^2+k^2-ik-jk}u^iv^jw^k}{(q;q)_i(q^2;q^2)_j(q^2;q^2)_k}.
\end{align}
Summing over $i$ and $j$ first using \eqref{Euler-2}, we deduce that
\begin{align}\label{F-start}
F(u,v,w)=\sum_{k=0}^\infty \frac{q^{k^2}w^k}{(q^2;q^2)_k}(-uq^{\frac{1}{2}-k};q)_\infty (-vq^{1-k};q^2)_\infty.
\end{align}

(1) From \eqref{F-start} we have
\begin{align}
&F(q^{1/2},q,q^{-1})=\sum_{k=0}^\infty \frac{q^{k^2-k}}{(q^2;q^2)_k}(-q^{1-k};q)_\infty (-q^{2-k};q^2)_\infty \nonumber \\
&=(-q;q)_\infty \sum_{k=0}^\infty \frac{q^{(k^2-k)/2}(-1;q)_k}{(q^2;q^2)_k}(-q^{2-k};q^2)_\infty \nonumber \\
&=(-q;q)_\infty (S_0(q)+S_1(q)). \label{S-split}
\end{align}
Here $S_0(q)$ and $S_1(q)$ correspond to the sums with even and odd values of $k$, respectively.

We have
\begin{align}
&S_0(q)=\sum_{k=0}^\infty \frac{q^{2k^2-k}(-1;q)_{2k}(-q^{2-2k};q^2)_\infty}{(q^2;q^2)_{2k}} \nonumber \\
&=(-q^2;q^2)_\infty \sum_{k=0}^\infty \frac{q^{k^2}(-1;q)_{2k}(-1;q^2)_k}{(q^2;q^2)_{2k}} \nonumber \\
&=\frac{1}{2}\left(3\frac{(q^2;q^2)_\infty^2}{(q;q)_\infty^2}-(q;q)_\infty^2 \right). \quad \text{(by \eqref{key-id-1})} \label{S0-result}
\end{align}
Similarly,
\begin{align}
&S_1(q)=\sum_{k=0}^\infty \frac{q^{2k^2+k}(-1;q)_{2k+1}(-q^{1-2k};q^2)_\infty}{(q^2;q^2)_{2k+1}} \nonumber \\
&=(-q;q^2)_\infty \sum_{k=0}^\infty \frac{q^{k^2+k}(-1;q)_{2k+1}(-q;q^2)_k}{(q^2;q^2)_{2k+1}} \nonumber \\
&=\frac{1}{2}\left(3\frac{(q^2;q^2)_\infty^2}{(q;q)_\infty^2}+(q;q)_\infty^2  \right). \quad \text{(by \eqref{key-id-3})} \label{S1-result}
\end{align}
Substituting \eqref{S0-result} and \eqref{S1-result} into \eqref{S-split}, we obtain \eqref{thm-id-1}.

(2) From \eqref{F-start} we have
\begin{align}
&F(q^{1/2},q,1)=\sum_{k=0}^\infty \frac{q^{k^2}}{(q^2;q^2)_k}(-q^{1-k};q)_\infty (-q^{2-k};q^2)_\infty \nonumber \\
&=(-q;q)_\infty \sum_{k=0}^\infty \frac{q^{(k^2+k)/2}(-1;q)_k(-q^{2-k};q^2)_\infty}{(q^2;q^2)_k} \nonumber \\
&=(-q;q)_\infty (S_0(q)+S_1(q)). \label{T-split}
\end{align}
Here $S_0(q)$ and $S_1(q)$ correspond to the sums with even and odd values of $k$, respectively.

We have
\begin{align}
&S_0(q)=\sum_{k=0}^\infty \frac{q^{2k^2+k}(-1;q)_{2k}(-q^{2-2k};q^2)_\infty}{(q^2;q^2)_{2k}} \nonumber \\
&=(-q^2;q^2)_\infty \sum_{k=0}^\infty \frac{q^{k^2+2k}(-1;q)_{2k}(-1;q^2)_k}{(q^2;q^2)_{2k}} \nonumber \\
&=\frac{1}{2}\left(\frac{(q^2;q^2)_\infty^2}{(q;q)_\infty^2}+(q;q)_\infty^2  \right). \quad \text{(by \eqref{key-id-2})} \label{T0-result}
\end{align}
Similarly,
\begin{align}
&S_1(q)=\sum_{k=0}^\infty \frac{q^{2k^2+3k+1}(-1;q)_{2k+1}(-q^{1-2k};q^2)_\infty}{(q^2;q^2)_{2k+1}} \nonumber \\
&=(-q;q^2)_\infty \sum_{k=0}^\infty \frac{q^{k^2+3k+1}(-1;q)_{2k+1}(-q;q^2)_k}{(q^2;q^2)_{2k+1}} \nonumber \\
&=\frac{1}{2}\left(\frac{(q^2;q^2)_\infty^2}{(q;q)_\infty^2}-(q;q)_\infty^2  \right). \quad \text{(by \eqref{key-id-4})} \label{T1-result}
\end{align}
Substituting \eqref{T0-result} and \eqref{T1-result} into \eqref{T-split}, we obtain \eqref{thm-id-2}.
\end{proof}

Before we present the proof of the last theorem, we recall the following identity \cite[Lemma 2.2, Eq.\ (2.6)]{WW24-1}: for $n\geq 0$,
\begin{align}\label{id-transform}
\frac{q^{n(n-1)/2}}{(q;q)_n}=\sum_{i+j=n} \frac{q^{i^2+j^2-i}}{(q^2;q^2)_i(q^2;q^2)_j}.
\end{align}
Let $\eta(\tau):=q^{1/24}(q;q)_\infty$ ($q=e^{2\pi i\tau}$, $\mathrm{Im}~ \tau>0$) be the Dedekind eta function. It is well known that $\eta(\tau)$ is a modular form of weight $1/2$.

\begin{proof}[Proof of Theorem \ref{thm-counterexample}]

We can rewrite the sum $F(q^{b_1},q^{b_2},q^{b_3})$ defined in \eqref{F-defn} as
\begin{align}
&\sum_{n,k,\ell\geq 0} \frac{q^{\frac{1}{2}n^2+k^2+\ell^2-n\ell-k\ell+b_1n+b_2k+b_3\ell}}{(q;q)_n(q^2;q^2)_k(q^2;q^2)_\ell}
\nonumber \\
&=\sum_{i,j,k,\ell\geq 0} \frac{q^{i^2-i+j^2+k^2+\ell^2-(i+j)\ell -k\ell+(b_1+\frac{1}{2})(i+j)+b_2k+b_3\ell}}{(q^2;q^2)_i(q^2;q^2)_j(q^2;q^2)_k(q^2;q^2)_\ell}  \quad \text{(by \eqref{id-transform}) }\nonumber \\
&=\sum_{i,j,k,\ell \geq 0} \frac{q^{i^2+j^2+k^2+\ell^2-i\ell -j\ell -k\ell +(b_1-\frac{1}{2})i+(b_1+\frac{1}{2})j+b_2k+b_3\ell}}{(q^2;q^2)_i(q^2;q^2)_j(q^2;q^2)_k(q^2;q^2)_\ell}=f_{A,B,0}(q^2), \label{122-f}
\end{align}
where $A$ is stated in Theorem \ref{thm-counterexample} and
$$B=\Big(\frac{1}{2}b_1-\frac{1}{4},\frac{1}{2}b_1+\frac{1}{4},\frac{1}{2}b_2,\frac{1}{2}b_3\Big)^\mathrm{T}.$$
Setting $(b_1,b_2,b_3)=(1/2,1,-1)$ and $(1/2,1,0)$, we deduce from \eqref{thm-id-1}, \eqref{thm-id-2} and \eqref{122-f} that
\begin{align}
f_{A,B_1,1/16}(q^2)=3\frac{\eta^3(2\tau)}{\eta^3(\tau)}, \quad f_{A,B_2,1/16}(q^2)=\frac{\eta^3(2\tau)}{\eta^3(\tau)}.
\end{align}
 Hence $f_{A,B_i,1/16}(q)$ are modular for $i=1,2$.

In the same way, we can rewrite the sum of the type in \eqref{WW-id-1} and \eqref{WW-id-2} as
\begin{align}
&\sum_{n,k,\ell\ge 0}\frac{q^{\frac{3}{2}n^2+2k^2+4\ell^2+2nk+4n\ell+4k\ell+b_1n+b_2k+b_3\ell}}{(q;q)_n(q^2;q^2)_k(q^2;q^2)_\ell} \nonumber \\
&=\sum_{i,j,k,\ell\geq 0} \frac{q^{i^2+j^2-i+(i+j)^2+2k^2+4\ell^2+2(i+j)k+4(i+j)\ell+4k\ell +(b_1+\frac{1}{2})(i+j)+b_2k+b_3\ell} }{(q^2;q^2)_i(q^2;q^2)_j(q^2;q^2)_k(q^2;q^2)_\ell}  \quad \text{(by \eqref{id-transform}) } \nonumber \\
&=\sum_{i,j,k,\ell \geq 0} \frac{q^{2i^2+2j^2+2k^2+4\ell^2+2ij+2ik+2jk+4i\ell +4j\ell+4k\ell +(b_1-\frac{1}{2})i+(b_1+\frac{1}{2})j+b_2k+b_3\ell} }{(q^2;q^2)_i(q^2;q^2)_j(q^2;q^2)_k(q^2;q^2)_\ell} \nonumber \\
&=f_{A^\star,\widetilde{B},0}(q^2). \label{M-f}
\end{align}
Here $A^\star=A^{-1}$ and
\begin{align}
\widetilde{B}=\left(\frac{1}{2}b_1-\frac{1}{4},\frac{1}{2}b_1+\frac{1}{4},\frac{1}{2}b_2,\frac{1}{2}b_3\right)^\mathrm{T}.
\end{align}
In particular, setting $(b_1,b_2,b_3)=(1/2,1,0)$ and $(5/2,3,4)$ we see that the vector $\widetilde{B}$ are exactly $B_1^\star$ and $B_2^\star$ stated in \eqref{exam-data-2}. Hence from \eqref{WW-id-1}, \eqref{WW-id-2} and \eqref{M-f} we see that
\begin{align}
f_{A^\star,B_1^\star,C'}(q^2)=\frac{1}{4}q^{2C'-\frac{1}{24}}\left(3\frac{\eta(2\tau)}{\eta(\tau)}+\frac{\eta^3(\tau)}{\eta(2\tau)}\right), \label{dual-exp-1} \\
f_{A^\star,B_1^\star,C'}(q^2)=\frac{1}{4}q^{2C'-\frac{25}{24}}\left(\frac{\eta(2\tau)}{\eta(\tau)}-\frac{\eta^3(\tau)}{\eta(2\tau)}\right). \label{dual-exp-2}
\end{align}
This shows that $f_{A^\star,B_i^\star,C'}(q)$ ($i=1,2$) can essentially be written as a sum of two modular forms of weights 0 and 1, respectively. Therefore, it cannot be modular for any $C'$.
\end{proof}

\subsection*{Acknowledgements}
The author thanks Prof.\ Haowu Wang and Boxue Wang for some helpful discussions. This work was supported by the National Key R\&D Program of China (Grant No.\ 2024YFA1014500).

\end{document}